\documentclass[11pt]{amsart}
\usepackage[colorlinks,citecolor=green,dvipdfm]{hyperref}

\newtheorem{theorem}{Theorem}[section]
\newtheorem{lemma}[theorem]{Lemma}

\theoremstyle{definition}
\newtheorem{definition}[theorem]{Definition}
\newtheorem{question}[theorem]{Question}

\newtheorem{corollary}[theorem]{Corollary}
\newtheorem{remark}[theorem]{Remark}

\theoremstyle{remark}

\newcommand{\be}{\begin{equation}}
\newcommand{\ee}{\end{equation}}

\numberwithin{equation}{section}

\begin{document}

\title{Cauchy-Schwarz-type inequalities on K\"{a}hler manifolds-II}

\author{Ping Li}
\address{%
Department of Mathematics\\
Tongji University\\
Siping Road 1239, Shanghai 200092\\
China}

\email{pingli@tongji.edu.cn}
\thanks{The author was partially
supported by the National Natural Science Foundation of China (Grant
No. 11471247) and and the Fundamental Research Funds for the Central
Universities.}




\begin{abstract}
We establish in this note some Cauchy-Schwarz-type inequalities on
compact K\"{a}hler manifolds, which generalize the classical
Khovanskii-Teissier inequalities to higher-dimensional cases. Our
proof is to make full use of the mixed Hodge-Riemann bilinear
relations due to Dinh and Nguy$\hat{\text{e}}$n. A proportionality
problem related to our main result is also proposed.
\end{abstract}

\maketitle

\section{Introduction and main results}
Suppose $X$ is an $n$-dimensional algebraic manifold and
$D_1,D_2,\ldots,D_n$ are $n$ (not necessarily distinct) ample
divisors on $X$. Then we have the following opposite
Cauchy-Schwarz-type inequality, \be\label{KT}([D_1D_2D_3\cdots
D_n])^2\geq[D_1D_1D_3\cdots D_n]\cdot[D_2D_2D_3\cdots D_n],\ee where
$[\cdot]$ denotes the intersection number of the divisors inside it
and the equality holds if and only if the two divisors $D_1$ and
$D_2$ are numerically proportional.

(\ref{KT}) was discovered independently by Khovanskii and Teissier
around in 1979 (\cite{Kh},\cite{Tei}) and now is called
Khovanskii-Teissier inequality. This equality is indeed a
generalization of the classical Aleksandrov-Fenchel inequalities and
thus present a nice relationship between the theory of mixed volumes
and algebraic geometry (\cite[p. 114]{Fu}). The proof of (\ref{KT})
is to apply the usual Hodge-Riemann bilinear relations (\cite[p.
122-123]{GH}) to the K\"{a}hler classes determined by these divisors
and an induction argument. The approach also suggests that the usual
Hodge-Riemann bilinear relations may be extended to the mixed case.
After some partial results towards this direction
(\cite{Gr},\cite{Ti}), this aim was achieved in its full generality
by Dinh and Nguy$\hat{\text{e}}$n in \cite{DN}.

We would like to point out a fact, which was not mentioned
explicitly in \cite{DN}, that (\ref{KT}) can now be extended by the
mixed Hodge-Riemann bilinear relations as follows. Suppose $\omega,
\omega_1, \omega_2,\ldots,\omega_{n-2}$ are $n-1$ K\"{a}hler classes
and $\alpha\in H^{1,1}(M,\mathbb{R})$ an \emph{arbitrary}
real-valued $(1,1)$-form on an $n$-dimensional compact connected
K\"{a}hler manifold $M$. Then we have
\be\label{CS1}\big(\int_M\alpha\wedge\omega\wedge\omega_1\wedge\cdots\wedge\omega_{n-2}\big)^2\geq
\big(\int_M\alpha^2\wedge\omega_1\wedge\cdots\wedge\omega_{n-2}\big)
\cdot\big(\int_M\omega^2\wedge\omega_1\wedge\cdots\wedge\omega_{n-2}\big),\ee
where the equality holds if and only if $\alpha\in\mathbb{R}\omega$.
Indeed, \cite[Theorem A]{DN} tells us that the index of the
following bilinear form \be\label{bilinear}Q(u,v):=\int_M u\wedge
v\wedge\omega_1\wedge\cdots\wedge\omega_{n-2}, \qquad u,v\in
H^{1,1}(M,\mathbb{R}),\ee
 is of the form $(+,-,\cdots,-)$, i.e., the
positive and negative indices are $1$ and $h^{1,1}-1$ respectively,
where $h^{1,1}$ is the corresponding Hodge number of $M$ \big(the
dimension of $H^{1,1}(M,\mathbb{R})$\big). Define a real-valued
function $f(t):=Q(\omega+t\alpha,\omega+t\alpha)$
$(t\in\mathbb{R})$. Then $f(0)>0$ as
$\omega,\omega_1,\cdots,\omega_{n-2}$ are all K\"{a}hler classes and
thus their product is strictly positive. $\omega+t\alpha$ spans a
$2$-dimensional subspace in $H^{1,1}(M,\mathbb{R})$ if $\alpha$ and
$\omega$ are linearly independent and thus $f(t_0)<0$ for some
$t_0\in\mathbb{R}$ in view of the index of $Q(\cdot,\cdot)$. Then
the discriminant of $f(t)$ gives (\ref{CS1}) with strict sign
$``>"$.

When these K\"{a}hler classes are all equal:
$\omega=\omega_1=\cdots=\omega_{n-2}$, (\ref{CS1}) degenerates to
the following special case:
\be\label{CS2}\big(\int_M\alpha\wedge\omega^{n-1}\big)^2\geq
\big(\int_M\alpha^2\wedge\omega^{n-2}\big)\cdot\big(\int_M\omega^{n}\big),\qquad
\forall~\alpha\in H^{1,1}(M,\mathbb{R}),\ee which is quite
well-known and, to the author's best knowledge, should be due to
Apte in \cite{Ap}. Inspired by (\ref{CS2}), the author asked in
\cite{Li} whether or not there exists a similar inequality to
(\ref{CS2}) for those $\alpha\in H^{p,p}(M,\mathbb{R})$ $(1\leq
p\leq [\frac{n}{2}])$ and obtained a related result (\cite[Theorem
1.3]{Li}), whose proof is also based on the usual Hodge-Riemann
bilinear relations. As an application we presented some Chern number
inequalities when the Hodge numbers of the manifolds satisfy some
constraints (\cite[Corollary 1.5]{Li}). Now keeping the mixed
Hodge-Riemann bilinear relations established in \cite{DN} in mind,
we may also ask if the main idea of the proof in \cite{Li} can be
carried over to the mixed case to extend the $\alpha$ in (\ref{CS1})
to $H^{p,p}(M,\mathbb{R})$ for $1\leq p\leq [\frac{n}{2}]$. The
answer is affirmative and this is the main goal of our current
article. So this article can be viewed as a sequel to \cite{Li},
which explains its title either.

Our main result (Theorem \ref{mainresult}) will be stated in the
rest of this section. In Section \ref{section2} we briefly review
the mixed Hodge-Riemann bilinear relations and then present the
proof of Theorem \ref{mainresult}. In Section \ref{section3} we
discuss a proportionality problem related to (\ref{KT}) posed by
Teissier and propose a similar problem related to our main result.

In order to state our result as general as possible, we would like
to investigate the elements in $H^{p,p}(M,\mathbb{C})$, i.e.,
complex-valued $(p,p)$-forms on $M$. The following definition is
inspired by (\ref{CS1}) and is a mixed analogue to \cite[Definition
1.1]{Li}.

\begin{definition}\label{def}
Suppose $M$ is an $n$-dimensional compact connected K\"{a}hler
manifold. For $1\leq p\leq[\frac{n}{2}]$, $\alpha\in
H^{p,p}(M,\mathbb{C})$ and $n-2p+1$ K\"{a}hler calsses
$\omega,\omega_1,\ldots,\omega_{n-2p}$, we put
$\Omega_p:=\omega_1\wedge\cdots\wedge\omega_{n-2p}$ and define \be
g(\alpha,\omega;\Omega_p):=\big(\int_M\alpha\wedge\bar{\alpha}\wedge\Omega_p\big)
\cdot\big(\int_M\omega^{2p}\wedge\Omega_p\big)-
\big(\int_M\alpha\wedge\omega^p\wedge\Omega_p\big)\cdot\big(\int_M\bar{\alpha}
\wedge\omega^p\wedge\Omega_p\big).\nonumber\ee
 $\alpha$ is said to satisfy Cauchy-Schwarz (resp.
opposite Cauchy-Schwarz) inequality with respect to the K\"{a}hler
classes $\omega$ and $(\omega_1,\ldots,\omega_{n-2p})$ if
$g(\alpha,\omega;\Omega_p)\geq 0$ (resp.
$g(\alpha,\omega;\Omega_p)\leq 0$).
\end{definition}

\begin{remark}
Note that $\alpha\in H^{p,p}(M,\mathbb{R})$ if and only if
$\alpha=\bar{\alpha}$. Also note that $g(\alpha,\omega;\Omega_p)$ in
the above definition is a real number and so we can discuss its
non-negativity or non-positivity.
\end{remark}

The main result of this note, which extends \cite[Theorem 1.3]{Li}
to the mixed case, is the following

\begin{theorem}\label{mainresult}
Suppose $M$ is an $n$-dimensional compact connected K\"{a}hler
manifold.
\begin{enumerate}
\item
Given $1\leq p\leq[\frac{n}{2}]$, all elements in
$H^{p,p}(M,\mathbb{C})$ satisfy Cauchy-Schwarz inequality with
respect to any K\"{a}hler classes $\omega$ and
$(\omega_1,\ldots,\omega_{n-2p})$ (in the sense of Definition
\ref{def}) if and only if the Hodge numbers of $M$ satisfy
\be\label{hodge1}h^{2i,2i}=h^{2i+1,2i+1},\qquad 0\leq
i\leq[\frac{p+1}{2}]-1.\ee

\item
All elements in $H^{1,1}(M,\mathbb{C})$ satisfy opposite
Cauchy-Schwarz inequality with respect to any K\"{a}hler classes
$\omega$ and $(\omega_1,\ldots,\omega_{n-2p})$.

\item
Given $2\leq p\leq[\frac{n}{2}]$, all elements in
$H^{p,p}(M,\mathbb{C})$ satisfy opposite Cauchy-Schwarz inequality
with respect to any K\"{a}hler classes $\omega$ and
$(\omega_1,\ldots,\omega_{n-2p})$ if and only if the Hodge numbers
of $M$ satisfy \be\label{hodge2}h^{2i-1,2i-1}=h^{2i,2i},\qquad 1\leq
i\leq[\frac{p}{2}].\ee
\end{enumerate}
Moreover, in all the cases mentioned above, the equalities hold if
and only if these $\alpha$ are proportional to $\omega^p.$
\end{theorem}

The first part of the following corollary extends the
Khovanskii-Teissier inequalities (\ref{KT}) and (\ref{CS1}).
\begin{corollary}\label{coro1}
\begin{enumerate}~
\item
\be\begin{split}
&\big(\int_M\alpha\wedge\omega\wedge\omega_1\wedge\cdots\wedge\omega_{n-2}\big)\cdot
\big(\int_M\bar{\alpha}\wedge\omega\wedge\omega_1\wedge\cdots\wedge\omega_{n-2}\big)\\
\geq&
\big(\int_M\alpha\wedge\bar{\alpha}\wedge\omega_1\wedge\cdots\wedge\omega_{n-2}\big)
\cdot\big(\int_M\omega^2\wedge\omega_1\wedge\cdots\wedge\omega_{n-2}\big)
\end{split}\nonumber\ee for any K\"{a}hler classes
$\omega,\omega_1,\ldots,\omega_{n-2}$ and any $\alpha\in
H^{1,1}(M,\mathbb{C})$, where the equality holds if and only if
$\alpha\in\mathbb{C}\omega$.

\item
If $h^{1,1}=1$, then all elements in $H^{2,2}(M,\mathbb{C})$ ($n\geq
4$) satisfy Cauchy-Schwarz inequality with respect to any K\"{a}hler
classes in the sense of Definition \ref{def}. Moreover, the equality
case holds if and only if the element is proportional to $\omega^2$.

\item
If $h^{1,1}=h^{2,2}$, then all elements in $H^{2,2}(M,\mathbb{C})$
$(n\geq 4)$ and $H^{3,3}(M,\mathbb{C})$ $(n\geq 6)$ satisfy opposite
Cauchy-Schwarz inequality with respect to any K\"{a}hler classes in
the sense of Definition \ref{def}. Moreover, the equality case holds
if and only if the element is proportional to $\omega^2$ or
$\omega^3$ respectively.
\end{enumerate}
\end{corollary}

\begin{remark}
In \cite[Example 1.7]{Li}, the author described in detail many
examples of compact connected K\"{a}hler manifolds whose Hodge
numbers satisfy $h^{1,1}=1$ and $h^{1,1}=h^{2,2}$ respectively.
These include the complete intersections in complex projective
spaces, the complex flag manifolds $G/P_{\textrm{max}}$ ($G$ is a
semisimple complex Lie group and $P_{\textrm{max}}$ is a maximal
parabolic subgroup of $G$), the one point blow-up of complex
projective spaces and so on.
\end{remark}

\section{Proof of the main result}\label{section2}
\subsection{The mixed Hodge-Riemann bilinear relations}
In this subsection we briefly recall the mixed Hodge-Riemann biliner
relations established in \cite{DN} by Dinh and
Nguy$\hat{\text{e}}$n.

As before denote by $M$ an $n$-dimensional compact connected
K\"{a}hler manifold.  We arbitrarily fix two non-negative integers
$p,q$ such that $p,q\leq[\frac{n}{2}]$ and $n-p-q+1$ K\"{a}hler
classes $\omega,\omega_1,\ldots,\omega_{n-p-q}$ on $M$. Put
$\Omega:=\omega_1\wedge\cdots\wedge\omega_{n-p-q}$. Define the mixed
primitive subspace of $H^{p,q}(M,\mathbb{C})$ \emph{with respect to
$\omega$ and $\Omega$} by
\be\label{primitivedef}P^{p,q}(M;\omega,\Omega):=\{\alpha\in
H^{p,q}(M,\mathbb{C})~|~ \alpha\wedge\omega\wedge\Omega=0\}.\ee
Define the mixed Hodge-Riemann bilinear form
$Q_{\Omega}(\cdot,\cdot)$ with respect to $\Omega$ on
$H^{p,q}(M,\mathbb{C})$ by
\be\label{quadric}Q_{\Omega}(\alpha,\beta):=(\sqrt{-1})^{q-p}
(-1)^{\frac{(p+q)(p+q+1)}{2}}\int_M\alpha\wedge\bar{\beta}\wedge\Omega,
\qquad \alpha,\beta\in H^{p,q}(M,\mathbb{C}).\ee

 \begin{remark}
 \begin{enumerate}~
\item
 Note that this
definition of $Q_{\Omega}(\cdot,\cdot)$ differs from that in
(\ref{bilinear}) by a sign when $p=q=1$.

\item
Clearly when $\omega=\omega_1=\cdots\omega_{n-p-q}$,
$P^{p,q}(M;\omega,\Omega)$ and $Q_{\Omega}(\cdot,\cdot)$ are nothing
but the usual primitive cohomology group and Hodge-Riemann bilinear
form with respect to the K\"{a}hler class $\omega$.

\item
The symbols $P^{p,q}(M;\omega,\Omega)$ and $Q_{\Omega}(\cdot,\cdot)$
we use here are simply denoted by $P^{p,q}(M)$ and $Q(\cdot,\cdot)$
respectively in \cite{DN}. We use the current symbols to avoid
confusion as they stress the dependence on the choices of $\omega$
and $\Omega$, whose advantage will be clear in the process of our
proof in Theorem \ref{mainresult} in the next section.
\end{enumerate}
\end{remark}

With the above notation understood, we have the following remarkable
result due to Dinh and Nguy$\hat{\text{e}}$n in \cite[Theorems
A,B,C]{DN}, which extends the classical Hodge-Riemann biliner
relations.

\begin{theorem}[mixed Hodge-Riemann bilinear relations]\label{DN}
\begin{enumerate}~
\item
(mixed Hard Lefschetz theorem) The linear map
$$\tau:~H^{p,q}(M,\mathbb{C})\rightarrow H^{n-q,n-p}(M,\mathbb{C})$$
given by
\be\label{hardlef}\tau(\alpha):=\alpha\wedge\Omega,\qquad\alpha\in
H^{p,q}(M,\mathbb{C})\ee is an isomorphism.

\item
(mixed Lefschetz decomposition) We have the following canonical
decomposition:
\be\label{lefdec}H^{p,q}(M,\mathbb{C})=P^{p,q}(M;\omega,\Omega)\oplus\big(\omega\wedge
H^{p-1,q-1}(M,\mathbb{C})\big).\ee Here
$H^{p-1,q-1}(M,\mathbb{C}):=0$
 if either $p=0$ or $q=0$.
\item
(Positive-definiteness) The mixed Hodge-Riemann bilinear form
$Q_{\Omega}(\cdot,\cdot)$ is positive-definite on the mixed
primitive subspace $P^{p,q}(M;\omega,\Omega)$.
\end{enumerate}
\end{theorem}

\begin{remark}
Note that $P^{p,q}(M;\omega,\Omega)$ depends on $\omega$ and
$\Omega$ while $H^{p,q}(M,\mathbb{C})$ is clearly independent of
them. This means, if we fix $\omega$ but change
$\omega_1,\ldots,\omega_{n-p-q}$, then $\Omega$ is also changed
respectively and so is $P^{p,q}(M;\omega,\Omega)$. But the mixed
Lefschetz decomposition theorem tells us that (\ref{lefdec}) remains
true. So the reference K\"{a}hler classes
$\omega,\omega_1,\ldots,\omega_{n-p-q}$ in context should be clear
when we apply (\ref{lefdec}). For unambiguity we shall use the
sentence "We apply (\ref{lefdec}) to $\alpha\in
H^{p,q}(M,\mathbb{C})$ with respect to the reference K\"{a}hler
classes $\omega$ and $(\omega_1,\ldots,\omega_{n-2p})$" to emphasize
it. This notation will play a key role in the proof of (\ref{1star})
Lemma \ref{lemma1}.
\end{remark}

\subsection{Proof of Theorem \ref{mainresult}}
We now apply the mixed Hodge-Riemann bilinear relations to prove our
Theorem \ref{mainresult}.

The following lemma uses the full power of (\ref{lefdec}).
\begin{lemma}\label{lemma1}
\begin{enumerate}~
\item
\be\label{dimensionformula}\text{dim}_{\mathbb{C}}
P^{p,q}(M;\omega,\Omega)=h^{p,q}-h^{p-1,q-1},\qquad 0\leq
p,q\leq[\frac{n}{2}],\ee where $h^{p-1,q-1}:=0$ if either $p=0$ or
$q=0$. This means that the dimension of $P^{p,q}(M;\omega,\Omega)$
is independent of $\omega$ and $\Omega$ and only depends on the
complex structure of $M$.

\item
Let $\alpha\in H^{p,p}(M,\mathbb{C})$ with $1\leq
p\leq[\frac{n}{2}]$ and $\omega,\omega_1,\ldots,\omega_{n-2p}$ be
$n-2p+1$ K\"{a}hler classes. Put
$\Omega_p:=\omega_1\wedge\cdots\wedge\omega_{n-2p}$. Then this
$\alpha$ can be written as follows.
\be\label{1star}\alpha=\lambda\omega^p+\sum_{i=1}^p\alpha_i\wedge\omega^{p-i},\ee
 where $\lambda\in\mathbb{C}$ and
 \be\label{star}\alpha_i\in P^{i,i}(M;\omega,\omega^{2(p-i)}\wedge\Omega_p)\qquad\big(\Leftrightarrow\alpha_i\wedge
 \omega^{2(p-i)+1}\wedge\Omega_p=0\big).\ee

\item
$g(\alpha,\omega;\Omega_p)$ given in Definition \ref{def} has the
following expression in terms of $\alpha_i$:
\be\label{2star}\begin{split}g(\alpha,\omega;\Omega_p)=
\big(\int_M\omega^{2p}\wedge\Omega_p\big)\cdot\big(\sum_{i=1}^p\int_M\alpha_i\wedge\bar{\alpha_i}
\wedge\omega^{2(p-i)}\wedge\Omega_p\big).\end{split}\ee
\end{enumerate}
\end{lemma}
\begin{proof}~
\begin{enumerate}
\item

The usual Hard Lefschetz theorem (\cite[p. 122]{GH}) tells us that
the map
$$\omega^{n-p-q}\wedge(\cdot):~H^{p,q}(M,\mathbb{C})\rightarrow
H^{n-q,n-p}(M,\mathbb{C})$$ is an isomorphism. This means that, for
$p+q\leq n-1$, the map
$$\omega\wedge(\cdot):~H^{p,q}(M,\mathbb{C})\rightarrow
H^{p+1,q+1}(M,\mathbb{C})$$ is injective and consequently
$$\text{dim}_{\mathbb{C}}\big(\omega\wedge
H^{p-1,q-1}(M,\mathbb{C})\big)=h^{p-1,q-1}$$ for $1\leq
p,q\leq[\frac{n}{2}],$ which, together with (\ref{lefdec}), leads to
(\ref{dimensionformula}).

\item
The strategy for proving (\ref{1star}) is to apply (\ref{lefdec})
repeatedly to yield the desired $\alpha_i$.

We first apply (\ref{lefdec}) to this $\alpha$ with respect to the
reference K\"{a}hler classes $\omega$ and
$(\omega_1,\ldots,\omega_{n-2p})$ to yield $\alpha_p$:
\be\label{ref1}\alpha=\alpha_p+\omega\wedge
\tilde{\alpha}_{p-1}\qquad\text{with} \qquad\alpha_p\in
P^{p,p}(M;\omega,\Omega_p).\nonumber \ee
 We continue to apply
(\ref{lefdec}) to $\tilde{\alpha}_{p-1}\in
H^{p-1,p-1}(M,\mathbb{C})$ with respect to the reference K\"{a}hler
classes $\omega$ and $(\omega,\omega,\omega_1,\ldots,\omega_{n-2p})$
to yield $\alpha_{p-1}$:
\be\label{ref2}\tilde{\alpha}_{p-1}=\alpha_{p-1}+\omega\wedge\tilde{\alpha}_{p-2}\qquad\text{with}
\qquad\alpha_{p-1}\in
P^{p-1,p-1}(M;\omega,\omega^2\wedge\Omega_p).\nonumber\ee Obviously
the next step is to apply (\ref{lefdec}) to $\tilde{\alpha}_{p-2}\in
H^{p-2,p-2}(M,\mathbb{C})$ with respect to the reference K\"{a}hler
classes $\omega$ and
$(\omega,\omega,\omega,\omega,\omega_1,\ldots,\omega_{n-2p})$ to
obtain
\be\label{ref2}\tilde{\alpha}_{p-2}=\alpha_{p-2}+\omega\wedge\tilde{\alpha}_{p-3}~
\qquad\text{with\qquad $\alpha_{p-2}\in
P^{p-2,p-2}(M;\omega,\omega^4\wedge\Omega_p)$}.\nonumber\ee Now it
is easy to see that repeated use of (\ref{lefdec}) to
$\tilde{\alpha}_{p-i}$ determined by $\tilde{\alpha}_{p-i+1}$ with
respect to the K\"{a}hler classes $\omega$ and
$(\underbrace{\omega,\ldots,\omega}_{\text{$2i$
copies}},\omega_1,\ldots,\omega_{n-2p})$ yields the desired
$\alpha_{p-i}:$
\be\tilde{\alpha}_{p-i}=\alpha_{p-i}+\omega\wedge\tilde{\alpha}_{p-i-1}~
\qquad\text{with\qquad $\alpha_{p-i}\in
P^{p-i,p-i}(M;\omega,\omega^4\wedge\Omega_p)$}.\nonumber\ee
 This completes the proof of (\ref{1star}).

\item
We now know from (\ref{1star}) that
\be\label{3star}\alpha\wedge\omega^{p}\wedge\Omega_p=
\lambda\omega^{2p}\wedge\Omega_p+\sum_{i=1}^p\alpha_i\wedge\omega^{2p-i}\wedge\Omega_p\ee
and \be\label{4star}\begin{split}
&\alpha\wedge\bar{\alpha}\wedge\Omega_p\\
=&(\lambda\omega^p+\sum_{i=1}^p\alpha_i\wedge\omega^{p-i})\wedge
(\bar{\lambda}\omega^p+\sum_{j=1}^p\bar{\alpha}_j\wedge\omega^{p-j})\wedge\Omega_p\\
=& \big(|\lambda|^2\omega^{2p}+
\lambda\sum_{j=1}^p\bar{\alpha}_j\wedge\omega^{2p-j}
+\bar{\lambda}\sum_{i=1}^p\alpha_i\wedge\omega^{2p-i}+
\sum_{i,j=1}^p\alpha_i\wedge\bar{\alpha}_j\wedge\omega^{2p-(i+j)}\big)\wedge\Omega_p
.\end{split}\ee

Note that
\begin{eqnarray}\label{s}
\left\{ \begin{array}{ll}
\text{$\alpha_i\wedge\omega^{2(p-i)+1}\wedge\Omega_p=0$ by
(\ref{star})}, &
1\leq i\leq p,\\
2p-i\geq
2(p-i)+1, & 1\leq i\leq p,\\
2p-(i+j)\geq2(p-\text{max\{i,j\}})+1, & 1\leq i\neq j\leq p.
\end{array} \right.
\end{eqnarray}

This means that (\ref{3star}) and (\ref{4star}) can be simplified
via (\ref{s}) as follows.
\be\label{5star}\alpha\wedge\omega^{p}\wedge\Omega_p=
\lambda\omega^{2p}\wedge\Omega_p\ee and
\be\label{6star}\alpha\wedge\bar{\alpha}\wedge\Omega_p=
|\lambda|^2\omega^{2p}\wedge\Omega_p+
\sum_{i=1}^p\alpha_i\wedge\bar{\alpha}_i\wedge\omega^{2(p-i)}\wedge\Omega_p.\ee

Integrating (\ref{5star}) and (\ref{6star}) over $M$ deduces that
\be\label{star1}\lambda=\frac{\int_M\alpha\wedge\omega^{p}\wedge\Omega_p}
{\int_M\omega^{2p}\wedge\Omega_p},\qquad
\bar{\lambda}=\frac{\int_M\bar{\alpha}\wedge\omega^{p}\wedge\Omega_p}
{\int_M\omega^{2p}\wedge\Omega_p},\ee and
\be\label{star2}\int_M\alpha\wedge\bar{\alpha}\wedge\Omega_p=
\lambda\cdot\bar{\lambda}\cdot\int_M\omega^{2p}\wedge\Omega_p+
\sum_{i=1}^p\int_M\alpha_i\wedge\bar{\alpha}_i
\wedge\omega^{2(p-i)}\wedge\Omega_p.\ee (\ref{2star}) now follows
from substituting the two expressions in (\ref{star1}) for the
$\lambda$ and $\bar{\lambda}$ in (\ref{star2}).
\end{enumerate}
\end{proof}
Now we are ready to prove Theorem \ref{mainresult}, our main result
in this article.
\begin{proof}
It suffices to prove the first part in Theorem \ref{mainresult} as
the resulting two cases are similar.

Since $\alpha_i\in P^{i,i}(M;\omega,\omega^{2(p-i)}\wedge\Omega_p)$,
the positive-definiteness of the mixed Hodge-Riemann bilinear forms
 guarantees that
 $Q_{(\omega^{2(p-i)}\wedge\Omega_p)}(\alpha_i,\alpha_i)\geq0$ with
 equality if and only if $\alpha_i=0$. This, together with the
 definition of $Q_{(\cdot)}(\cdot,\cdot)$ in (\ref{quadric}),
 implies that
\be\label{8star}(-1)^i\int_M\alpha_i\wedge\bar{\alpha_i}\wedge\omega^{2(p-i)}\wedge\Omega_p\geq
0,\qquad 1\leq i\leq p,\ee with the equality holds if and only if
$\alpha_i=0$.

We first show the ``if" part of (1) in Theorem \ref{mainresult}.

 The dimension formula (\ref{dimensionformula}) in Lemma
\ref{lemma1} and the assumption (\ref{hodge1}) imply that
$$\alpha_{2i+1}=0,\qquad 0\leq i\leq[\frac{p+1}{2}]-1,$$ which,
together with (\ref{2star}) and (\ref{8star}), give us
$$g(\alpha,\omega;\Omega_p)=\big(\int_M\omega^{2p}\wedge\Omega_p\big)\cdot\big(\sum_{\substack{1\leq i\leq
p\\i~\textrm{even}}}\int_M\alpha_i\wedge\bar{\alpha_i}
\wedge\omega^{2(p-i)}\wedge\Omega_p\big)\geq 0,$$
 with equality if and only if all $\alpha_i=0$ and thus
$\alpha\in\mathbb{C}\omega^p$ by the decomposition formula
(\ref{1star}).

The proof of the ``only if" part.

Suppose on the contrary that there exists some $1\leq
i_0\leq[\frac{p+1}{2}]-1$ such that
$h^{2i_0+1,2i_0+1}>h^{2i_0,2i_0}$. Then we can choose a
$$0\neq\alpha(i_0)\in
P^{2i_0+1,2i_0+1}(M;\omega,\omega^{2\big(p-(2i_0+1)\big)}\wedge\Omega_p)$$
by (\ref{dimensionformula}) and set
$$\theta:=\omega^p+\alpha(i_0)\wedge\omega^{p-(2i_0+1)}.$$ But in
this case
$$g(\theta,\omega;\Omega_p)=\big(\int_M\omega^{2p}\wedge\Omega_p\big)\cdot
\big(\int_M\alpha(i_0)\wedge\bar{\alpha(i_0)}\wedge
\omega^{2\big(p-(2i_0+1)\big)}\wedge\Omega_p\big)<0$$ and thus this
$\theta$ does not satisfy Cauchy-Schwarz inequality with respect to
the K\"{a}hler classes $\omega$ and
$(\omega_1,\ldots,\omega_{n-2p})$ in the sense of Definition
\ref{def}, which contradicts to the assumption. This gives the
desired proof.
\end{proof}

The corollary below follows from the process of the above proof.

\begin{corollary}
Suppose $\alpha\in H^{p,p}(M,\mathbb{C})$ with $1\leq
p\leq[\frac{n}{2}]$. Then this $\alpha$ satisfies Cauchy-Schwarz
(resp. opposite Cauchy-Schwarz) inequality with respect to
K\"{a}hler classes $\omega,\omega_1,\ldots,\omega_{n-2p}$ if those
$\alpha_i$ with $i$ odd (resp. even) determined by the decomposition
formula (\ref{2star}) all vanish.
\end{corollary}

\section{A proportionality problem}\label{section3}
Let $\mathcal{K}\in H^{1,1}(M,\mathbb{R})$ be the K\"{a}hler cone of
$M$, which consists of all the K\"{a}hler classes of $M$. Recall
that $c\in H^{1,1}(M,\mathbb{R})$ is called a \emph{nef} class if
$c\in\bar{\mathcal{K}}$, the closure of the K\"{a}hler cone. So nef
classes can be approximated by K\"{a}hler classes. This means
Theorem \ref{mainresult} has  the following corollary.

\begin{corollary}\label{coro2}
Suppose $M$ is an $n$-dimensional compact connected K\"{a}hler
manifold.
\begin{enumerate}
\item
Given $1\leq p\leq[\frac{n}{2}]$, if the Hodge numbers of $M$
satisfy \be h^{2i,2i}=h^{2i+1,2i+1},\qquad 0\leq
i\leq[\frac{p+1}{2}]-1,\nonumber\ee then for any $\alpha\in
H^{p,p}(M,\mathbb{C})$ and any nef classes $c,c_1,\ldots,c_{n-2p}$
we have
 \be\label{nefcs1}\big(\int_M\alpha\wedge\bar{\alpha}\wedge C_p\big)\cdot
\big(\int_M c^{2p}\wedge C_p\big)\geq \big(\int_M\alpha\wedge
c^p\wedge C_p\big)\cdot\big(\int_M\bar{\alpha} \wedge c^p\wedge
C_p\big),\ee where $C_p=c_1\wedge\cdots\wedge c_{n-2p}.$

\item
For any $\alpha\in H^{1,1}(M,\mathbb{C})$ and any nef classes
$c,c_1,\ldots,c_{n-2}$ we have
 \be\label{nefcs2}\big(\int_M\alpha\wedge\bar{\alpha}\wedge C)\cdot
\big(\int_M c^{2p}\wedge C\big)\leq \big(\int_M\alpha\wedge
c^p\wedge C\big)\cdot\big(\int_M\bar{\alpha} \wedge c^p\wedge
C\big),\ee where $C=c_1\wedge\cdots\wedge c_{n-2}.$

\item
Given $2\leq p\leq[\frac{n}{2}]$, if the Hodge numbers of $M$
satisfy \be h^{2i-1,2i-1}=h^{2i,2i},\qquad 1\leq
i\leq[\frac{p}{2}],\nonumber\ee then for any $\alpha\in
H^{p,p}(M,\mathbb{C})$ and any nef classes $c,c_1,\ldots,c_{n-2p}$
we have
 \be\label{nefcs3}\big(\int_M\alpha\wedge\bar{\alpha}\wedge C_p\big)\cdot
\big(\int_M c^{2p}\wedge C_p\big)\leq \big(\int_M\alpha\wedge
c^p\wedge C_p\big)\cdot\big(\int_M\bar{\alpha} \wedge c^p\wedge
C_p\big),\ee where $C_p=c_1\wedge\cdots\wedge c_{n-2p}.$
\end{enumerate}
\end{corollary}

 However, unlike Theorem \ref{mainresult} for K\"{a}hler classes, we can \emph{not}
conclude directly in this case that the equalities in
(\ref{nefcs1}), (\ref{nefcs2}) and (\ref{nefcs3}) hold if and only
if $\alpha$ and the nef class $c$ are proportional. So a natural
question is to characterize the equalities in
(\ref{nefcs1})-(\ref{nefcs3}), a very special case of which has been
proposed by Teissier in \cite{Tei1} as a further question related to
his inequality (\ref{KT}) and we shall briefly reivew it in what
follows.

 (\ref{KT}) or (\ref{nefcs2}) gives us that,
for two \emph{nef} divisors $D_1$ and $D_2$ on an algebraic manifold
and for $1\leq k\leq n-1$, we have \be\label{nefCS2}\big([D_1^k
D_2^{n-k}]\big)^2\geq[D_1^{k-1} D_2^{n-k+1}]\cdot[D_1^{k+1}
D_2^{n-k-1}].\ee Teissier considered in \cite{Tei1} that how to
characterize the equality case in (\ref{nefCS2}) for nef and
\emph{big} divisors $D_1$ and $D_2$ (recall that a nef divisor $D$
is called \emph{big} if moreover $[D^n]>0$). This problem was solved
in \cite[Theorem D]{BFJ} by Boucksom, Favre and Jonsson, whose
result asserts that for two nef and big divisors $D_1$ and $D_2$ the
equality in (\ref{nefCS2}) holds if and only if $D_1$ and $D_2$ are
numerically proportional. Very recently Fu and Xiao obtained the
same type result in the context of K\"{a}hler manifolds
(\cite[Theorem 2.1]{FX2}), some of whose ideas are based on their
previous work in \cite{FX1}.

\begin{remark}
The expression used in \cite[Theorem D, (2)]{BFJ} is slightly
different from our (\ref{nefCS2}) but they are indeed equivalent
(see, for instance, the equivalent statements in (\cite[Theorem
2.1]{FX2}).
\end{remark}

With the Cauchy-Schwarz-type inequalities in its full generality in
Corollary \ref{coro2} in hand, we can now end our article by posing
the following problem, whose solution is obviously beyond the
content of this note.

\begin{question}
How to characterize the three inequality cases in (\ref{nefcs1}),
(\ref{nefcs2}) and (\ref{nefcs3})? Clearly $\alpha$ being
proportional to $c^p$ is a sufficient condition. Is this also a
necessary condition? Or weakly how to establish such a necessary
condition by imposing more constraints on  the element $\alpha$, the
nef classes $c,c_1,\ldots,c_{n-2p}$ and/or the underlying K\"{a}hler
manifold $M$?
\end{question}

\section*{Acknowledgements}
The author thanks Ji-Xiang Fu for drawing his attention to the
reference \cite{DN} in a talk on \cite{FX2} in a workshop held in
Center of Mathematical Science, Zhejiang University, inspired by
which the author realized that his idea in \cite{Li} could be
carried over to extend the classical Khovanskii-Teissier
inequalities to higher-dimensional cases.

\bibliographystyle{amsplain}

\end{document}